\documentclass{amsart}
\usepackage{latexsym,amsfonts}

\newtheorem{theorem}{Theorem}

\theoremstyle{definition} 
\newtheorem{example}[theorem]{Example}
\def\co{\colon\thinspace}

\DeclareMathOperator{\Leib}{Leib}

\begin{document}

\title{On Levi's Theorem for Leibniz algebras} 
\author{Donald W. Barnes}
\address{1 Little Wonga Rd, Cremorne NSW 2090 Australia}
\email{donwb@iprimus.com.au}


\begin{abstract}  A Lie algebra over a field of characteristic $0$ splits over its soluble radical and all complements are conjugate.  I show that the splitting theorem extends to Leibniz algebras but that the conjugacy theorem does not.
\end{abstract}

\maketitle

Let $L$ be a finite-dimensional left Leibniz algebra over a field of characteristic $0$.  I denote by $d_a$ the left multiplication operator $d_a\co L \to L$ defined by $d_a(x) = ax$ for all $a, x \in L$.

I call the subspace $\langle x^2 \mid x \in L \rangle$ spanned by the squares of elements of $L$ the Leibniz kernel of $L$ and denote it $\Leib(L)$.  It is an abelian ideal of $L$, $L/\Leib(L)$ is a Lie algebra  and $\Leib(L)L = 0$.  Let $R =R(L)$ be the soluble radical of $L$.  Then $\Leib(L) \subseteq R$ .

\begin{theorem}  There exists a semi-simple subalgebra $S$ of $L$ such that $S + R = L$ and $S \cap R = 0$.
\end{theorem}

\begin{proof}  Put $K = \Leib(L)$.  By Levi's Theorem (see Jacobson \cite[Chapter III, p. ~91]{Jac}), there exists a semi-simple subalgebra $S^*/K$ of $L/K$ such $S^* + R = L$ and $S^* \cap R = K$.  It is sufficient to prove that $S^*$ splits over $K$, so we may suppose $R = K$.

Since $KL=0$, $L$ may be considered as a left module for the semi-simple Lie algebra $L/K$.  By Whitehead's Theorem (see Jacobson \cite[Chapter III, Theorem 8, p.79]{Jac}), this module is completely reducible.  Thus there exists a submodule $S$ complementing $K$.  Since $LS  \subseteq S$, we have $SS \subseteq S$ and $S$ is a subalgebra.
\end{proof}

\begin{example}  Let $S$ be a simple Lie algebra and let $K$ be isomorphic to $S$ as left $S$-module.  I denote by $x'$ the element of $K$ corresponding to $x \in S$ under this isomorphism.  I make $K$ into a Leibniz module by defining the right action to be $0$.  Let $L$ be the split extension of $K$ by $S$.  Then $L$ is a Leibniz algebra and $\Leib(L) = K$.  The space $S_1 =\{(s,s')\mid s \in S\}$ is another subalgebra complementing $K$ since, using the module isomorphism, we have
$$(s,s')(t,t') = (st, st') = (st, (st)').$$ 
For any $x = (s,k) \in L$, the inner derivation $d_x = d_s$, $d_x(S) \subseteq S$ and so also $\exp(d_x)(S) \subseteq S$.  Thus $S$ and $S_1$ are not conjugate.
\end{example}

\bibliographystyle{amsplain}

\end{document}